\documentclass[%
preprint%
]{elsarticle}

\usepackage{amsmath,amssymb}
\usepackage{mathtools}
\usepackage{amsthm}
\newtheorem{theorem}{Theorem}
\newtheorem{proposition}{Proposition}

\theoremstyle{definition}
\newtheorem{example}{Example}

\usepackage[thinlines]{easybmat}
\usepackage{array}

\newcommand{\SL}[1]{\ensuremath{\mathrm{SL}_{#1}}}
\newcommand{\Z}{\mathbb{Z}}
\newcommand{\Q}{\mathbb{Q}}
\newcommand{\bx}[1]{\raisebox{2pt}{\fcolorbox{black}{#1}{\rule{0pt}{2pt}\rule{2pt}{0pt}}}}
\usepackage{xcolor}
\definecolor{one}{RGB}{255,85,85}
\definecolor{nonz}{RGB}{255,255,0}
\definecolor{mone}{RGB}{255,170,85}

\usepackage{hyperref}

\begin{document}

\begin{frontmatter}

\title{Wildest \SL2-tilings}
\author{Andrei Zabolotskii}
\ead{andrei.zabolotskii@open.ac.uk}
\affiliation{
organization={School~of~Mathematics~and~Statistics, The~Open~University},
city={Milton~Keynes},
postcode={MK7~6AA},
country={United~Kingdom}
}

\begin{abstract}
Tame \SL2-tilings are related to Farey graph and friezes; much less is known about wild (not tame) \SL2-tilings. In this note, we demonstrate \SL2-tilings that are maximally wild: we prove that the maximum wild density of an integer \SL2-tiling is $\tfrac25$ and present \SL2-tilings over $\Z/N\Z$ with wild density~1.
\end{abstract}


\end{frontmatter}

\section{Introduction}
An \emph{\SL2-tiling} over a unital, commutative ring $R$ is a bi-infinite matrix with elements $m_{i,j}\in R$ indexed by $i,j\in\Z$, such that for all $i,j$,
\begin{equation}
\det
\begin{pmatrix}
m_{i,j} & m_{i,j+1} \\
m_{i+1,j} & m_{i+1,j+1}
\end{pmatrix}
= 1.
\label{eq:sl2}
\end{equation}

An \SL2-tiling is called \emph{tame} if, for all $i,j$,
\begin{equation}
\det
\begin{pmatrix}
m_{i\!-\!1\!,j\!-\!1} & m_{i\!-\!1\!,j} & m_{i\!-\!1,\!j\!+\!1} \\
m_{i,\phantom{\!\!-\!1}j\!-\!1} & m_{i,\phantom{\!\!\!-1}j} & m_{i,\phantom{\!\!\!-1}j\!+\!1} \\
m_{i\!+\!1\!,j\!-\!1} & m_{i\!+\!1\!,j} & m_{i\!+\!1\!,j\!+\!1} \\
\end{pmatrix}
 = 0.
\label{eq:tame}
\end{equation}
An \SL2-tiling that is not tame is called \emph{wild}. An entry $m_{i,j}$ of a given \SL2-tiling such that Eq.~\ref{eq:tame} does not hold will be called a \emph{wild entry}.

\begin{example}
\label{ex:unit}
Our basic example of a tame \SL2-tiling is the anti-periodic tiling of the integer plane by the $2\times2$ identity matrix.
\[
\begin{BMAT}(e){ccccccccc}{ccccccccc}
& & & & \vdots & & & & \\
& 0 & 1 & 0 & -1 & 0 & 1 & 0 & \\
& -1 & 0 & 1 & 0 & -1 & 0 & 1 & \\
& 0 & -1 & 0 & 1 & 0 & -1 & 0 & \\ \cdots
& 1 & 0 & -1 & 0 & 1 & 0 & -1 & \cdots \\
& 0 & 1 & 0 & -1 & 0 & 1 & 0 & \\
& -1 & 0 & 1 & 0 & -1 & 0 & 1 & \\
& 0 & -1 & 0 & 1 & 0 & -1 & 0 & \\
& & & & \vdots & & & & \\
\end{BMAT}
\]
\end{example}

\SL2-tilings were introduced in \cite{AsReSm2010}. Their classification into tame and wild was introduced in \cite{slk}.
Tame \SL2-tilings have been described in terms of paths in the Farey graph \cite{sl2short,OUmodular}.
\SL2-tilings bounded by two diagonals of zeros are known as the frieze patterns, or friezes. They are related to the combinatorics of cluster algebras \cite{AsReSm2010} as well as to many others topics. However, these relations have only been explored essentially for the case of positive frieze patterns, which are always tame.

Wild \SL2-tilings are much less understood. After \cite{slk}, there is only one work that considers them \cite{wildCuntz}. In it, a graph that helps to describe \SL k-friezes is introduced; for $k=2$, the vertices of that graph are all possible rows of a frieze pattern and edges connect the rows which can be adjacent in a frieze.

With tame \SL2-tilings being well understood, we turn to the other end of the spectrum and ask: what is the wildest possible \SL2-tiling, maximally different from tame ones? There are many ways to give this question a precise meaning. Some of them rely on a combinatorial model for arbitrary \SL2-tilings generalising that of the tame tilings, based on the Farey graph; such model will be presented elsewhere. Here we focus on the simplest possible interpretation: what is the greatest possible \emph{wild density} of an \SL2-tiling $M$, defined as
\begin{equation}
\label{eq:density}
\limsup_{r\to\infty}\frac{\text{number of wild entries in $M$ within distance $r$ from the origin}}{\text{total number of entries in $M$ within distance $r$ from the origin}}.
\end{equation}

For example, the \SL2-tiling shown in \cite[Eq.~(3)]{slk} has wild density $\tfrac14$.

The answers to our problem are given by Theorem~\ref{thm:max25} and Example~\ref{ex:max1}.

\section{The wildest integer \SL2-tiling}

Let $\begin{pmatrix}a&b&c\\d&e&f\\g&h&i\end{pmatrix}$ be a $3\times3$ block in an \SL2-tiling over the integers, $\Z$. Then it follows from Dodgson's condensation and the definition of an \SL2-tiling that
\begin{equation}
\label{eq:dod}
e\cdot\det\begin{pmatrix}a&b&c\\d&e&f\\g&h&i\end{pmatrix}=0.
\end{equation}
Therefore, if $e$ is a wild entry, it must necessarily be equal to~0. For this reason, within this section, we will refer to wild entries as \emph{wild zeros} and to other zero entries as \emph{tame zeros}.

Furthermore, in any \SL2-tiling,
\begin{equation}
\label{eq:det3}
\det\begin{pmatrix}a&b&c\\d&e&f\\g&h&i\end{pmatrix}=(a+c+g+i)+(cg-ai)e.
\end{equation}
Thus, if $e=0$ then the determinant equals the sum of the four corner entries of the matrix $a+c+g+i$.

Any zero entry, be it wild or tame, is necessarily surrounded by 1s and $-1$s in one of the two possible configurations:
\[
\begin{BMAT}(e){ccc}{ccc}
\cdot&1&\cdot\\
-1&0&1\\
\cdot&-1&\cdot
\end{BMAT}
\quad\text{or}\quad
\begin{BMAT}(e){ccc}{ccc}
\cdot&-1&\cdot\\
1&0&-1\\
\cdot&1&\cdot
\end{BMAT}
\]
Also, of the four diagonally adjacent neighbours of a wild zero, at least one must be nonzero, otherwise the $3\times3$ determinant becomes zero due to Eq.~\ref{eq:det3}, and the central zero entry becomes tame.

The set $\Z^2$ of indices of entries of an \SL2-tiling is a subset of the plane $\mathbb{R}^2$, which makes it natural to associate each tiling entry with an integer point on the plane, or with a square tile centred at that point. We will use this both in the reasoning (in particular, Eq.~\eqref{eq:density} refers to the distance in this sense) and the visualisation, always assigning the colour black to the tiles with a wild entry in the middle and other colours to other tiles. The wild density of a given \SL2-tiling is the same as the density of the colour black in the plane.

\pagebreak
\begin{example}
\label{ex:wildestZ}
The tiling from Example~\ref{ex:unit} modified by adding the nonzero elements $a_j$ (left) visualised as yellow tiles (right), arranged in an index-10 square sublattice of $\Z^2$. Thus, the coloured pattern is doubly-periodic. The \SL2-tiling itself, however, can be non-periodic, since each yellow tile corresponds to an arbitrary nonzero entry $a_j$, independent of the others. All zeros are wild.

\vskip2mm
\noindent
\begin{tabular}{m{0.5\textwidth}m{0.4\textwidth}}
\(
\small
\begin{BMAT}(e)[1pt]{cccccccccccccc}{cccccccccccccc}
& & & &&&& \rotatebox{90}{...}
&&& & & & \\
& 0 & -\!1 & 0 & 1 & 0 & -\!1 & \raisebox{0.3ex}{$a_1$} & 1 & 0 & -\!1 & 0 & 1 & \\
& 1 & 0 & -\!1 & \raisebox{0.3ex}{$a_2$} & 1 & 0 & -\!1 & 0 & 1 & 0 & -\!1 & 0 & \\
& \raisebox{0.3ex}{$a_3$} & 1 & 0 & -\!1 & 0 & 1 & 0 & -\!1 & 0 & 1 & \raisebox{0.3ex}{$a_4$} & -\!1 & \\
& -\!1 & 0 & 1 & 0 & -\!1 & 0 & 1 & \raisebox{0.3ex}{$a_5$} & -\!1 & 0 & 1 & 0 & \\
& 0 & -\!1 & 0 & 1 & \raisebox{0.3ex}{$a_6$} & -\!1 & 0 & 1 & 0 & -\!1 & 0 & 1 & \\ ...
& 1 & \raisebox{0.3ex}{$a_7$} & -\!1 & 0 & 1 & 0 & -\!1 & 0 & 1 & 0 & -\!1 & \raisebox{0.3ex}{$a_8$} & ... \\
& 0 & 1 & 0 & -\!1 & 0 & 1 & 0 & -\!1 & \raisebox{0.3ex}{$a_9$} & 1 & 0 & -\!1 & \\
& -\!1 & 0 & 1 & 0 & -\!1 & \raisebox{0.3ex}{$a_{\!1\!0}$} & 1 & 0 & -\!1 & 0 & 1 & 0 & \\
& 0 & -\!1 & \raisebox{0.3ex}{$a_{\!1\!1}$} & 1 & 0 & -\!1 & 0 & 1 & 0 & -\!1 & 0 & 1 & \\
& 1 & 0 & -\!1 & 0 & 1 & 0 & -\!1 & 0 & 1 & \raisebox{0.3ex}{$a_{\!1\!2}$} & -\!1 & 0 & \\
& 0 & 1 & 0 & -\!1 & 0 & 1 & \raisebox{0.3ex}{$a_{\!1\!3}$} & -\!1 & 0 & 1 & 0 & -\!1 & \\
& -\!1 & 0 & 1 & \raisebox{0.3ex}{$a_{\!1\!4}$} & -\!1 & 0 & 1 & 0 & -\!1 & 0 & 1 & 0 & \\
& & & &&&& \rotatebox{90}{...}
&&& & & & \\
\end{BMAT}
\)
&
\raisebox{-.56\height}{
\includegraphics[height=0.425\columnwidth]{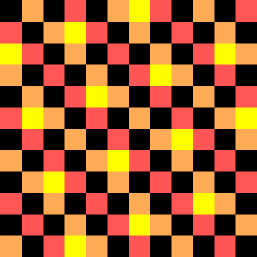}
}
\\
&
\small
~~\mbox{\bx{one}~$1$ \quad \bx{mone}~$-\!1$ \quad \bx{nonz}~$\ne0$ \quad \bx{black}~$0$ (wild)}
\end{tabular}

\end{example}

\begin{theorem}
\label{thm:max25}
The maximal wild density of an \SL2-tiling over the integers is $\tfrac25$, which is achieved by Example~\ref{ex:wildestZ}.
\end{theorem}
\begin{proof}
Due to the necessary conditions on the neighbourhood of a wild zero, the 4 immediately adjacent tiles and at least 1 diagonally adjacent tile of every black tile are not black. Then the wild zero in that black tile must be surrounded by the image of the following pentagonal shape under rotation by a multiple of $\pi/2$, without overlapping with the pentagons from other wild zeros:
\begin{figure}[h!]
\centering
\includegraphics[height=0.2\columnwidth]{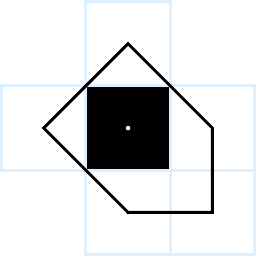}
\end{figure}

This pentagon contains the entire black square, the rest is filled with white (representing any non-black colour), and the density of black is $\tfrac25$. Therefore, the density of black in the entire plane cannot be greater than this value.

On the other hand, in an \SL2-tiling of the form shown in Example~\ref{ex:wildestZ}, the pentagonal neighbourhoods of wild zeros tessellate the plane without gaps, forming a Cairo pentagonal tessellation, attaining the upper bound for the wild density:

\pagebreak
\begin{figure}[h!]
\centering
\includegraphics[height=0.3\columnwidth]{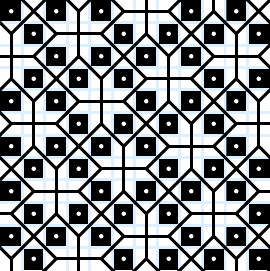}
\end{figure}

\end{proof}

In the rest of this section, we will treat Example~\ref{ex:wildestZ} as a specific \SL2-tiling over $\Q(a_1,a_2,\ldots)$ with $a_j$ being different formal variables.

Note that tameness also manifests through bounded rank: in a~tame \SL2-tiling no block has rank more than~2, when interpreted as a matrix over $\Q$.
In fact, the tilings shown in Example~\ref{ex:wildestZ} is ``maximally wild'' in this sense too: the rank of any block in it is almost maximal.

\begin{proposition}
In the \SL2-tiling over $\Q(a_1,a_2,\ldots)$ shown in Example~\ref{ex:wildestZ}, any block has rank deficiency at most~2.
\end{proposition}
\begin{proof}
Let $M$ be the \SL2-tiling from Example~\ref{ex:wildestZ}.
For any given odd (resp.\ even) nonnegative integer $n$, there are only 4 (resp.\ 3) equivalence classes of $n\times n$ blocks in $M$ under relabelling the variables $a_j$, rotations, reflections, and flipping the signs of all $\pm1$s. For example, the four inequivalent $5\times5$ blocks are shown below.
\[
\begin{BMAT}(e)[1pt]{ccccc}{ccccc}
1 & 0 & -\!1 & a_1 & 1  \\
 a_2 & 1 & 0 & -\!1 & 0  \\
-\!1 & 0 & 1 & 0 & -\!1  \\
0 & -\!1 & 0 & 1 & a_3 \\
1 & a_4 & -\!1 & 0 & 1
\end{BMAT}
\qquad
\begin{BMAT}(e)[1pt]{ccccc}{ccccc}
0 & -\!1 & a_1 & 1 & 0 \\
1 & 0 & -\!1 & 0 & 1  \\
 0 & 1 & 0 & -\!1 & 0  \\
-\!1 & 0 & 1 & a_2 & -\!1  \\
a_3 & -\!1 & 0 & 1 & 0
\end{BMAT}
\qquad
\begin{BMAT}(e)[1pt]{ccccc}{ccccc}
1 & 0 & -\!1 & 0 & 1  \\
0 & 1 & 0 & -\!1 & 0  \\
-\!1 & 0 & 1 & a_1 & -\!1  \\
a_2 & -\!1 & 0 & 1 & 0 \\
1 & 0 & -\!1 & 0 & 1
\end{BMAT}
\qquad
\begin{BMAT}(e)[1pt]{ccccc}{ccccc}
0 & -\!1 & 0 & 1 & 0 \\
1 & 0 & -\!1 & 0 & 1  \\
 0 & 1 & a_1 & -\!1 & 0  \\
-\!1 & 0 & 1 & 0 & -\!1  \\
0 & -\!1 & 0 & 1 & 0
\end{BMAT}
\]
They have rank deficiencies resp.\ 0, 0, 1, and 2. The last case ensures that the bound in the proposition is sharp.

In each of the three inequivalent $10\times10$ blocks (which are all present in the $12\times12$ block shown in Example~\ref{ex:wildestZ}, e.g.\ as the top-left, central, and bottom-right $10\times10$ subblocks), there is exactly one $a_j$ in each row and in each column.
Consider a block $M_0$ in $M$ of size at least $10\times10$. Represent it as a block matrix $M_0=\left(
\begin{array}{c|c}
A &B \\ \hline
C &D
\end{array}\right)$ with $A$ being $10\times10$; label the 10 elements of the form $a_j$ in $A$ as $a_1,\ldots,a_{10}$. For any minor in $D$ with a nonzero minor determinant $d$, there is another minor in $M_0$ formed by including the topmost 10 rows and the leftmost 10 columns of $M_0$ which is also nonzero because its minor determinant has a term proportional to $a_1\cdot\ldots\cdot a_{10}$ with coefficient $\pm d\ne0$. Therefore, the rank deficiency of $M_0$ is no greater than the rank deficiency of $D$.

Taking any block in $M$, we can repeatedly use this argument to shave off the topmost 10 rows and the leftmost 10 columns (not lowering the rank deficiency in the process) until a block with less than 10 rows or columns is left. We can replace it by any of its maximal square subblocks, as this will not lower the rank deficiency. It only remains to verify that any $n\times n$ block in $M$ with $n<10$ has rank deficiency at most~2. By the exhaustive check of inequivalent cases (there is a small number of them, as noted above), this is true.

\end{proof}

\section{The wildest \SL2-tiling ever}
As the previous section shows, one cannot obtain a wild density greater than $\tfrac25$ in an \SL2-tiling over the integers, or in fact over any integral domain. Now we consider \SL2-tilings over a unital, commutative ring $R$, possibly containing zero divisors. Suppose we are free to choose the ring $R$, trying to maximise the wild density. Then the following example solves the problem.

\begin{example}
\label{ex:pqrs}
Let $p,q,r,s$ be integers greater than 1 such that $ps-qr=1$. Let $N=pqrs$, $\alpha=qr-1$ and $\beta=ps-1$. Consider the ring $R=\Z/N\Z$ and the following bi-infinite matrix formed by repeating a $4\times4$ block periodically in both directions, where each entry is taken modulo $N$ and interpreted as an element of $R$:
\[
\begin{BMAT}(e){l1rrrr1rrrr1c}{t1cccc1cccc1c}
&&&&\vdots&&&&&\\
& p & q & -p & -q & p & q & -p & -q & \\
& r & s & -r & -s & r & s & -r & -s& \\
& \alpha p & \beta q & p & q& \alpha p & \beta q & p & q&\\
\ldots& \beta r & \alpha s & r & s& \beta r & \alpha s & r & s&\ldots\\
& p & q & -p & -q & p & q & -p & -q & \\
& r & s & -r & -s & r & s & -r & -s &\\
& \alpha p & \beta q & p & q& \alpha p & \beta q & p & q&\\
& \beta r & \alpha s & r & s& \beta r & \alpha s & r & s&\\
&&&&\vdots&&&&&
\end{BMAT}
\]
This is an \SL2-tiling. Moreover, each $3\times3$ determinant is equal to one of $pqr,pqs,prs,qrs$ (as elements of $R$) and thus is nonzero in $R$. That is, this is an \SL2-tiling over $R$ with wild density equal to~1.
\end{example}

The smallest possible $N$ in Example~\ref{ex:pqrs} is $N=72$, achieved for $p=s=3$, $q=2$, $r=4$. However, wild density~1 can be achieved even for a smaller $N$ if the \SL2-tiling is slightly modified, as our final example shows.

\begin{example}
\label{ex:max1}
In the following \SL2-tiling over $\Z/36\Z$, each entry is wild.
\[
\begin{BMAT}(e){c1cccc1cccc1c}{t1cccc1cccc1c}
& &&& \vdots&& & & & \\
& 3 & 2 & 33 & 34 & 3 & 2 & 33 & 34 & \\
& 4 & 3 & 32 & 33 & 4 & 3 & 32 & 33 & \\
& 9 & 16 & 3 & 2 & 9 & 16 & 3 & 2& \\
\ldots& 14 & 9 & 4 & 3 & 14 & 9 & 4 & 3& \ldots\\
& 3 & 2 & 33 & 34 & 3 & 2 & 33 & 34 & \\
& 4 & 3 & 32 & 33 & 4 & 3 & 32 & 33 & \\
& 9 & 16 & 3 & 2 & 9 & 16 & 3 & 2& \\
& 14 & 9 & 4 & 3 & 14 & 9 & 4 & 3& \\
&&& & \vdots &&& & & \\
\end{BMAT}
\]
\end{example}

\section*{Acknowledgements}
This research is supported by EPSRC grant EP/W524098/1.
I am grateful to Ian Short, Matty van Son, and Oleg Karpenkov for useful discussions and to Fran\c{c}ois Bergeron for pointing out a relevant reference.

\bibliography{wildest}{}
\bibliographystyle{plain}

\end{document}